\documentclass[11pt]{article}
\usepackage{graphicx} 
\usepackage{amsmath, amsthm, amsfonts, mathtools, physics}
\mathtoolsset{showonlyrefs} 
\usepackage{comment}
\usepackage{fullpage}
\usepackage{caption}
\usepackage{subcaption}
\usepackage{amssymb}
\usepackage{tikz}
\usetikzlibrary{matrix,backgrounds}
\tikzset{classification/.style={thick,yscale=.8,font=\small}}

\usepackage{float}
\usepackage{enumerate}

\theoremstyle{definition}

\newtheorem*{dfn*}{Definition}
\newtheorem{thm}{Theorem}
\newtheorem*{thm*}{Theorem}

\newtheorem{lemma*}{Lemma}

\newtheorem{prop}{Proposition}

\usepackage{tikz}
\usetikzlibrary{matrix,backgrounds}
\tikzset{classification/.style={thick,x=1.0cm,y=1.0cm,font=\small}}

\usepackage{physics,mathtools}
\usepackage{thmtools,thm-restate}
\newcommand{\OPT}{\mathrm{OPT}}

\DeclareMathOperator*{\argmin}{arg\,min}

\usepackage{authblk}
\title{The Last Success Problem with Samples}
\author[1]{Toru Yoshinaga\thanks{yoshinaga-toru106@g.ecc.u-tokyo.ac.jp}}
\author[1]{Yasushi Kawase\thanks{kawase@mist.i.u-tokyo.ac.jp}}
\affil{The University of Tokyo, Japan}
\date{}

\begin{document}
\maketitle
\begin{abstract}
The last success problem is an optimal stopping problem that aims to maximize the probability of stopping on the last success in a sequence of independent $n$ Bernoulli trials. In the classical setting where complete information about the distributions is available, Bruss~\cite{B00} provided an optimal stopping policy that ensures a winning probability of $1/e$. However, assuming complete knowledge of the distributions is unrealistic in many practical applications. This paper investigates a variant of the last success problem where samples from each distribution are available instead of complete knowledge of them. When a single sample from each distribution is allowed, we provide a deterministic policy that guarantees a winning probability of $1/4$. This is best possible by the upper bound provided by Nuti and Vondr\'{a}k~\cite{NV23}. Furthermore, for any positive constant $\epsilon$, we show that a constant number of samples from each distribution is sufficient to guarantee a winning probability of $1/e-\epsilon$.
\end{abstract}

\section{Introduction}
Imagine you are driving down a street toward a movie theater, hoping to park in a space near the destination. You cannot know in advance which parking spaces are free. Each time you encounter an available space, you must decide whether to park there or continue driving. Your goal is to maximize the probability of parking in the available space nearest to the theater without going back. This scenario is studied as \emph{the last success problem} in the literature on optimal stopping theory~\cite{D70,T82}.

The last success problem is an optimal stopping problem that aims to maximize the probability of stopping on the last \emph{success} in a sequence of independent $n$ Bernoulli trials, where the success probability of $i$th trial is known to be $p_i$ for each $i\in\{1,2,\dots,n\}$. The event of stopping on the last success is referred to as a \emph{win}, whereas the other event is referred to as a \emph{loss}.

This problem is a generalization of various optimal stopping problems, including the classical secretary problem, and has a wide range of applications such as parking~\cite{T82}, maintenance planning for production equipment~\cite{ILT07,LTI08}, and ethical choices for clinical trials~\cite{B18}. The last success problem was originally studied by Hill and Krengel~\cite{HK92}. Subsequently, Bruss~\cite{B00} provided an optimal stopping policy for this problem. Bruss's policy is to stop at the first success from the index $\tau=\min\big\{t\in\{1,2,\dots,n\}\mid \sum_{i=t+1}^n r_i<1\big\}$, where $r_i$ denotes the \emph{odds} defined as $p_i/(1-p_i)$. Here, we treat $r_i$ as $\infty$ if $p_i=1$. In words, this policy stops at the first observed success within the range where the sum of the odds of future successes is less than one. Bruss's policy guarantees a winning probability of at least $1/e$ if the sum of the odds is at least one, i.e., $\sum_{i=1}^nr_i\ge 1$, and this lower bound is shown to be optimal~\cite{B03}. Note that an assumption about the sum of the odds is needed to guarantee a winning probability. Without such an assumption, there can be a situation where no successes occur (i.e., $p_1=\dots=p_n=0$), leading to an inevitable loss of any policy.

However, the assumption that the success probabilities are completely known is too restrictive in certain cases. For example, in the parking problem, it seems unnatural to know the probability of each parking place being available. Nevertheless, if you have previously driven to the movie theater, you would be aware of whether each parking space was available at that time. To model such a situation, it is appropriate to assume unknown distributions along with finite data obtained from them. To address real-world applications, optimal stopping problems under limited information have recently been extensively studied~\cite{AMM21,AFGS22,AKW14,CCES24,CDFKZ20,CDFS19,DFLLR21,DK19,FHL21,GKSW24,KNR20,KNR22,MST21}.

In this paper, we investigate the last success problem in the sample model where the decision maker has only $m$ samples for each (unknown) distribution. Very little is known about the problem of this variant, excluding the fact that no policy achieves a winning probability strictly better than $1/4$ in the single sample case where $m=1$.

\begin{thm}[Nuti and Vondr\'{a}k~\cite{NV23}]\label{thm:NV}
For any positive real $\epsilon$, there does not exist a (randomized) stopping policy for the single sample last success problem that guarantees a winning probability of $1/4 + \epsilon$ even when a success occurs with probability one.
\end{thm}
Nuti and Vondr\'{a}k~\cite{NV23} provided this impossibility result while examining a setting named the \emph{adversarial order single sample secretary problem}, which is an optimal stopping problem with the goal of maximizing the probability of stopping on the largest value in a sequence of $n$ independent real-valued random variables. Observing a tentatively maximum value in this problem can be interpreted as observing a success in our problem. They also proved that the following simple policy achieves a winning probability of $1/4$: set the largest value from the samples as a threshold and stop at the first index whose value beats this threshold. However, this fact does not directly imply the existence of a stopping policy that attains a winning probability of $1/4$ for the last success problem. Indeed, for the last success problem, the winning probability for any algorithm is $0$ if the success probabilities are $0$. Note that, in the adversarial order single sample secretary problem, such a case does not exist because the first variable must be a tentatively maximum value.

\paragraph*{Our Contributions}
We first examine the no-sample model, where $m=0$, and show that no policy can guarantee a positive winning probability even when a success occurs with probability one (Theorem~\ref{thm:noinfo}).

For the single sample model, where $m=1$, we propose two natural deterministic policies, which we call the \emph{from the last success (FLS)} policy and the \emph{after the second last success (ASLS)} policy, and evaluate their winning probabilities. Each of the two policies first determines a threshold index from the sample sequence and then stops at the first success from the index. The FLS policy and the ASLS policy select the index of the last success and the \emph{next} index of the second-last success in the sample sequence as the threshold index, respectively. We show that the FLS policy cannot guarantee a winning probability of better than $(1-e^{-4})/4\approx 0.2454$ even in instances where a success occurs with probability one (Theorem~\ref{thm:FLS}). On the other hand, we prove that the ASLS policy guarantees a winning probability of $1/4$ for every instance where the sum of the odds is at least $(\sqrt{3}-1)/2\approx 0.3660$ (Theorem~\ref{thm:ASLS}). We also demonstrate that the ASLS policy is nearly optimal for any case with a restriction on the sum of the odds.

Moreover, we analyze a \emph{randomized} policy derived from a policy proposed by Nuti and Vondr\'{a}k~\cite{NV23} for the adversarial order single sample secretary problem. This policy randomly selects the threshold index to be either the index of the last success or the next index of the last success in the sample sequence, each with a probability of $1/2$. We conducted this analysis for comparison with our policies. Our analysis reveals that this policy guarantees a winning probability of $1/4$ if the sum of the odds is at least $1/2$ (Theorem~\ref{thm:FLSR}) for our problem setting. However, it fails to guarantee a winning probability of $1/4$ if the sum of the odds is slightly less than $1/2$. Thus, we conclude that the ASLS policy is superior in both performance and deterministic nature.

For the multiple sample model, we provide a policy that guarantees a winning probability of $1/e-\epsilon$ with a constant $m$ for any positive constant $\epsilon$ if the sum of the odds is at least $1$ (Theorem~\ref{thm:mltsample}). Notably, this result does not depend on the number of trials $n$. A natural policy would be to estimate the probability of success for each distribution and apply Bruss's policy. However, such a method may lead to errors that depend on the number of trials $n$. Instead, our policy determines a threshold index from the samples by estimating an index $i$ such that $\prod_{k=i}^n(1-p_k)\le 1/e+\delta$ and $\prod_{k=i+1}^n(1-p_k)\ge 1/e$ for a small $\delta>0$. To evaluate the winning probability of this policy, we first demonstrate that its winning probability is at least $1/e-\delta$ if it correctly identifies an index $i$ that satisfies these conditions. Subsequently, we show that an index $i$ meeting the conditions can be selected with high probability by utilizing a martingale property. Moreover, we prove that no policy can guarantee a winning probability of exactly $1/e$ if $m$ is a finite number (Proposition~\ref{prop:incomplete}).

\subsection{Related work}
One of the most fundamental problems in optimal stopping is the \emph{secretary problem} (see the survey by Ferguson~\cite{F89} for a detailed history). In the classical setting, applicants are interviewed one by one in a random order. After each interview, the interviewer must make an immediate and irrevocable decision to either hire or reject the candidate. The interviewer can only rank the candidates among those interviewed up to that point. The goal is to maximize the probability of hiring only the best candidate. The secretary problem can be viewed as a special case of the last success problem of known distributions where $p_i=1/i$ for $i=1,2,\dots,n$. This is because the occurrence of a best candidate so far in the secretary problem is analogous to the occurrence of a success in the last success problem.

Another fundamental problem is the \emph{prophet inequality problem}. In this problem, a decision-maker observes a sequence of individual real-valued random variables, one by one. For each observation, an irrevocable decision must be made to either select the current variable or wait for the next one. The objective is to select the variable with the best value. It is known that a $1/2$-competitive algorithm exists for this problem, and it is best possible~\cite{KS78}.

Models that allow limited sample access to unknown distributions have been proposed for online optimization in recent years. Azar et al.~\cite{AKW14} introduced a sample model in which inputs are drawn from unknown distributions, but the decision maker can access some samples from each distribution. They applied this model to the prophet inequalities under constraints, guaranteeing a constant competitive ratio for each setting. Rubinstein et al.~\cite{RWW20} proposed a stopping policy with a competitive ratio of $1/2$ for the single sample prophet inequality. Remarkably, this competitive ratio is optimal even when the distributions are fully known in advance~\cite{KS78}. Building on these works, online optimization problems with a limited number of samples have been studied extensively for the past decade~\cite{CDFFLPPR22,CDFS19,CCES20,CALTA20,CZ24,DRY15,DLPV23,GLT22,KNR20,NV23}.

Several studies have also been conducted on the last success problem in the unknown distribution setting. Bruss and Louchard~\cite{BL09} considered a variant of the last success problem where the distributions are unknown but identical (i.e., $p_1=\dots=p_n=p$). They investigated a stopping policy by estimating the success probability $p$ from observed random variables. The special case of our problem with $m=1$ can be reduced to the \emph{the single sample secretary problem in an adversarial order} by considering the no success scenario as a win (refer to the subsection~\ref{subsec:NV} in details). Nuti and Vondr\'{a}k~\cite{NV23} proposed a policy for the problem that ensures a winning probability of $1/4$ and showed that this is best possible by using Theorem~\ref{thm:NV}.\footnotemark{}\footnotetext{Nuti and Vondr\'{a}k~\cite{NV23} provided this lower bound even for a slightly more general problem called the \emph{adversarial order two-sided game of googol}.}

\section{Preliminaries}
\subsection{Model}
We formally define the last success problem with $m$ samples. For a positive integer $n$, we denote the set $\{1,\ldots,n\}$ by $[n]$. Suppose that there are $n$ random variables $X_1,\ldots, X_n$, following independent and non-identical Bernoulli distributions. A trial $i$ is called \emph{success} if $X_i=1$. For each trial $i\in[n]$, the \emph{success probability} $p_i=\Pr[X_i=1]$ is unknown. Instead, $m$ samples independently drawn from the $i$th Bernoulli distribution are available for each $i\in[n]$. We sequentially observe realizations of the variables $X_1,\ldots, X_n$. Upon observing a success at trial $i$, an immediate and irrevocable decision must be made to either halt the observation or continue to the subsequent trial. This decision must be based only on the observed values $X_1,\dots,X_i$ and the samples. The result is a \emph{win} if stopping the observation on the last success (i.e., $X_i=1$ and $X_{i+1}=\dots=X_n=0$). Note that, if no successes are observed by the end of the sequence, the result must be a loss. Our goal is to design a stopping policy that maximizes the probability of a win, that is, the probability of stopping at the last success.

We will evaluate policy performance with a worst-case analysis. However, for instances where successes never occur (i.e., $p_1=\dots=p_n=0$), the probability of a win is zero no matter what policies are used. To conduct a meaningful analysis, we use the \emph{sum of the odds} $R=\sum_{i=1}^n r_i$ as a parameter in our analysis, where $r_i=p_i/(1-p_i)$ is the odds of $i$th trial. This parameter $R$ plays a crucial role in the context of the last success problem~\cite{B00}. Let $\mathcal{I}_{m,R}$ be the set of instances of the last success problem with $m$ samples such that the sum of the odds is at least $R$. For a stopping policy $\mathcal{P}$ and an instance $I\in\mathcal{I}_{m,R}$, let $\mathcal{P}(I)$ be the winning probability when we apply $\mathcal{P}$ to $I$. Then, we call $\inf_{I\in\mathcal{I}_{m,R}}\mathcal{P}(I)$ the winning probability of $\mathcal{P}$ for \emph{the $(m,R)$-last success problem}. The $(m,R)$-last success problem is easier than the $(m,R')$-last success problem when $R>R'$ since $\mathcal{I}_R\subsetneq\mathcal{I}_{R'}$. Thus, the easiest case is the $(m,\infty)$-last success problem.

\subsection{Basic observation}
The sum of the odds $R$ is associated with the probability of no success as follows. 
\begin{restatable}{lemma}{nosuccess}
\label{lem:nosuccess}
For $p_1,\dots,p_n\in[0,1]$, let $R=\sum_{i=1}^n p_i/(1-p_i)$ and $Q=\prod_{i=1}^n (1-p_i)$. Then, we have $1/e^R\le Q\le 1/(1+R)$.
\end{restatable}
\begin{proof}
If $p_i=1$ for some $i$, we have $R=\infty$ and $Q=0$, and hence the desired inequalities hold.
Suppose that $p_i<1$ for all $i\in[n]$.
Then, we have 
\begin{align*}
Q&=\prod_{i=1}^n (1-p_i)
=\prod_{i=1}^n \frac{1}{1+\frac{p_i}{1-p_i}}
= \frac{1}{\prod_{i=1}^n\left(1+\frac{p_i}{1-p_i}\right)}.
\end{align*}
By considering binomial expansion, we get
$\prod_{i=1}^n\left(1+\frac{p_i}{1-p_i}\right)\ge 1+\sum_{i=1}^n \frac{p_i}{1-p_i} = 1+R$, and hence
\begin{align*}
Q=\frac{1}{\prod_{i=1}^n\left(1+\frac{p_i}{1-p_i}\right)}\le \frac{1}{1+R}.
\end{align*}
Moreover, we have
\begin{align*}
\prod_{i=1}^n\left(1+\frac{p_i}{1-p_i}\right)
&\le \left(\frac{1}{n}\sum_{i=1}^n\left(1+\frac{p_i}{1-p_i}\right)\right)^n
= \left(1+\frac{R}{n}\right)^n
\le e^R,
\end{align*}
where the first inequality follows from the AM-GM inequality, and the second inequality follows from the fact that $(1+1/x)^x\le e$ for any positive $x$. Thus, we obtain
\begin{align*}
Q=\frac{1}{\prod_{i=1}^n\left(1+\frac{p_i}{1-p_i}\right)}\ge \frac{1}{e^R}. 
\end{align*}
\end{proof}
We can derive an upper bound of the winning probability for the last success problem with samples by referencing the upper bound in the complete information setting.
\begin{prop}[Bruss~\cite{B00}]
\label{prop:bruss}
The winning probability of any (randomized) stopping policy for the $(m,R)$-last success problem is at most $R/e^R$ if $R\le 1$ and at most $1/e$ if $R\ge 1$.
\end{prop}

For the prophet inequality problem~\cite{KS78}, it has been demonstrated that the optimal performance of $1/2$, which can be achieved with complete information, can be attained using just a single sample~\cite{RWW20}. In contrast, for the last success problem, the optimal winning probability $1/e$, which can be achieved with complete information, cannot be attained using a finite number of samples.

\begin{restatable}{prop}{incomplete}\label{prop:incomplete}
For any (randomized) stopping policy and any positive integer $m$, the winning probability for the $(m,\infty)$-last success problem is strictly less than $1/e$.
\end{restatable}
\begin{proof}
Fix a policy $\mathcal{P}$ and a positive integer $m$. For a positive even number $n$, let $I_n$ and $I'_n$ be the instances of the $(m,\infty)$-last success problem such that the success probabilities $(p_1,p_2,p_3,\dots,p_n)$ are $(1,0,0,\dots,0)$ and $(1,2/n,2/n,\dots,2/n)$, respectively. Then, the probability is $1$ that the samples for $I_n$ are all successes for the first trial and all failures for the other trials. The samples for $I_n'$ are exactly the same with probability $(1-2/n)^{(n-1)m}$ because all $m(n-1)$ samples coming from all trials except the first must be $0$. Let $\alpha_n$ be the probability that the policy stops at the first trial for this outcome of the samples. Without loss of generality, we may assume that $\alpha_n\ge 1/e$ for any $n$ since otherwise the winning probability for $I_n$ is strictly smaller than $1/e$. If the policy stops at the first trial in the instance $I_n'$, the winning probability is $(1-2/n)^{n-1}$. Moreover, for $I_n'$, the winning probability of the policy is at most $(1-2/n)^{(n-2)/2}$ regardless of whether the policy stops at the first trial. This upper bound can be shown by the fact that Bruss's policy~\cite{B00} (i.e., stopping at the first observed success from the index $\tau=\min\{t\in[n]\mid\sum_{i=t+1}^n r_i<1\}$, where $r_i$ is the odds of $i$th trial) is optimal when the success probabilities are known in advance. As the odds $r_i$ of each trial $i\in[n]$ is $\infty$ if $i=1$ and $2/(n-2)$ if $i\ge 2$, Bruss's policy sets $\tau=n/2+2$ as the threshold by $\sum_{j=\tau}^n r_j=(n-\tau+1)\cdot\frac{2}{n-2}=\frac{n-2}{2}\cdot\frac{2}{n-2}=1$ and $\sum_{j=\tau+1}^n r_j(n-\tau)\cdot\frac{2}{n-2}=\frac{n-4}{2}\cdot\frac{2}{n-2}<1$. Therefore, the optimal winning probability for the instance $I_n'$ is 
\begin{align*}
    \sum_{i=\tau}^n \frac{p_i}{1-p_i}\prod_{j=\tau}^n(1-p_j)
    &=\left(\prod_{j=\tau}^n(1-p_j)\right)\cdot \sum_{i=\tau}^n \frac{p_i}{1-p_i}
    =\left(1-\frac{2}{n}\right)^{n-\tau+1}\cdot\sum_{i=\tau}^n r_i\\
    &=\left(1-\frac{2}{n}\right)^{\frac{n-2}{2}}\cdot 1
    =\left(1-\frac{2}{n}\right)^{\frac{n-2}{2}},
\end{align*}
and this value is an upper bound of the winning probability of $\mathcal{P}$ for $I_n'$.

We demonstrate that the winning probability of the policy $\mathcal{P}$ for $I_n'$ is strictly less than $1/e$ when $n$ is a sufficiently large number. If the samples are all successes for the first trial and all failures for the other trials (this happens with probability $(1-2/n)^{(n-1)m}$), the policy $\mathcal{P}$ wins with probability at most $\alpha_n\cdot\left(1-2/n\right)^{n-1}+(1-\alpha_n)\cdot(1-2/n)^{(n-2)/2}$ because it stops at the first trial with probability $\alpha_n$. If there is at least one success in a sample from a trial other than $1$ (this happens with probability $1-\left(1-2/n\right)^{(n-1)m}$), the winning probability of $\mathcal{P}$ for $I_n'$ is at most $(1-2/n)^{(n-2)/2}$. Therefore, the winning probability of $\mathcal{P}$ for $I'_n$ is at most 
\begin{align*}
\MoveEqLeft
\textstyle\!\left(1-\tfrac{2}{n}\right)^{(n-1)m}\!\left(\alpha_n\left(1-\tfrac{2}{n}\right)^{n-1}\!\!+\!(1-\alpha_n)\left(1-\tfrac{2}{n}\right)^{\frac{n-2}{2}}\right)+\left(1-\left(1-\tfrac{2}{n}\right)^{(n-1)m}\right)\left(1-\tfrac{2}{n}\right)^{\frac{n-2}{2}}\\
\textstyle&=\left(1-\tfrac{2}{n}\right)^{\frac{n-2}{2}}-\left(1-\tfrac{2}{n}\right)^{(n-1)m}\alpha_n\!\cdot\!\left(\left(1-\tfrac{2}{n}\right)^{\frac{n-2}{2}}-\left(1-\tfrac{2}{n}\right)^{n-1}\right)\\
\textstyle&\le \left(1-\tfrac{2}{n}\right)^{\frac{n-2}{2}}-\left(1-\tfrac{2}{n}\right)^{(n-1)m}\frac{1}{e}\cdot\left(\left(1-\tfrac{2}{n}\right)^{\frac{n-2}{2}}-\left(1-\tfrac{2}{n}\right)^{n-1}\right)\\
\textstyle&\to \frac{1}{e}-\frac{1}{e^{2(m+1)}}\left(1-\frac{1}{e}\right)
\end{align*}
as $n$ goes to infinity. Hence, the winning probability of any policy $\mathcal{P}$ is strictly less than $1/e$ for the instance of $I_n$ or $I_n'$ with a sufficiently large $n$.
\end{proof}

Of course, this proposition also holds for the $(m,R)$-last success problem with any positive real $R$.

For the no-sample setting where $m=0$, we can demonstrate that no (randomized) stopping policy can guarantee a positive constant probability of winning. Indeed, consider a scenario where the adversary selects an index $j$ uniformly at random from $[n]$ and sets the success probabilities as $p_1=\dots=p_j=1$ and $p_{j+1}=\dots=p_n=0$. In this scenario, we need to correctly guess the index $j$ to win, which can only be done with a probability of $1/n$. Hence, by Yao's principle~\cite{Y77}, any stopping policy can win with probability at most $1/n$ for an instance with $n$ trials. Note that the sum of the odds of the instance is $\infty$ by $p_1=1$. By considering the limit as $n$ approaches infinity, we obtain the following theorem.
\begin{thm}\label{thm:noinfo}
The winning probability of any (randomized) stopping policy is zero for the $(0,\infty)$-last success problem.
\end{thm}

\section{Single sample model}\label{sec:sumtheodds}
In this section, we concentrate on the special case where $m=1$ for the last success problem with $m$ samples, which we refer to as the \emph{single sample last success problem}. Throughout this section, we write $Y_i\in\{0,1\}$ to denote the sample of $i$th trial for each $i\in[n]$.

\subsection{Estimating the sum of the odds}\label{subsec:estimatingOdds}
When all success probabilities are known in advance, the optimal stopping policy for the last success problem is Bruss's policy~\cite{B00}, wherein one stops on the first success for which the sum of the odds for the future trials is at most $1$. Given this, a natural approach to our single sample setting is to compute an estimated success probability $\hat{p}_i$ for each trial $i\in[n]$ from the sample $Y_i$ and then apply Bruss's policy based on these estimated success probabilities. Let $\alpha_0$ and $\alpha_1$ be real numbers in the range $[0,1]$, and consider estimating $\hat{p}_i$ as $\alpha_0$ if $Y_i=0$ and $\hat{p}_i$ as $\alpha_1$ if $Y_i=1$. We will then show that regardless of the values of $\alpha_0$ and $\alpha_1$ used for estimation, we cannot achieve a winning probability of $1/4$ even when a success occurs with probability $1$.

We examine three cases: (i) $\alpha_0>0$, (ii) $\alpha_1<1/2$, and (iii) $\alpha_0=0$ and $\alpha_1\ge 1/2$.

For the first case, where $\alpha_0>0$, we consider an instance with $n=\lceil 1/a_1\rceil+1$ trials such that the success probabilities are $p_1=1$ and $p_2=\cdots=p_n=0$. Note that the last success is the first trial, which means that we must stop at the first trial to win. Here, the estimated success probabilities must be $\hat{p}_1=\alpha_1$ and $\hat{p}_2=\cdots=\hat{p}_n=\alpha_0$. As $\sum_{i=2}^n \hat{p}_i/(1-\hat{p}_i)=\lceil 1/\alpha_0\rceil\cdot \alpha_0/(1-\alpha_0)>1$, the threshold index must be at least $2$. Therefore, the policy wins with a probability of $0$.

For the second case, where $\alpha_1<1/2$, we examine an instance with $p_1=p_2=1$ and $n=2$. The estimated success probabilities must be $\hat{p}_1=\hat{p}_2=\alpha_1$. Further, the threshold index cannot be $2$ because $\hat{p}_2/(1-\hat{p}_2)=\alpha_1/(1-\alpha_1)<1$. Thus, the policy always stops at the first trial, resulting in a winning probability of $0$.

Now, it remains to examine the last case where $\alpha_0=0$ and $\alpha_1\ge 1/2$. In this case, the policy sets the index of the last success in the samples as a threshold because $\alpha_1/(1-\alpha_1)\ge 1$ and $\alpha_0/(1-\alpha_0)=0$. We refer to this policy as the \emph{From the Last Success (FLS)} policy, and we demonstrate that this policy cannot guarantee a winning probability exceeding $(1-e^{-4})/4\approx 0.2454<1/4$ even when at least one success is guaranteed to occur, i.e., the sum of the odds $R$ is positive infinity.
\begin{thm}\label{thm:FLS}
    The winning probability of the FLS policy for the $(1,\infty)$-last success problem is at most $(1-e^{-4})/4$.
\end{thm}
\begin{proof}
    Consider the instance where the success probabilities are $p_1=1$ and $p_2=\dots=p_n=2/n$. Note that at least one success appears in the trials since $p_1=1$. Recall that the FLS policy sets the threshold to the index of the last observed success in the sampling sequence. We calculate the probability for the instance that each index becomes the threshold, as well as the winning probability for each threshold.

    In the FLS policy, the threshold is set to $\tau=1$ when no success is observed in the sample sequence except for the first trial, i.e., $Y_2=Y_3=\dots=Y_n=0$. This happens with probability $(1-2/n)^{n-1}$. When $\tau=1$, the FLS policy wins if and only if there is no success in the actual sequence except for the first trial, i.e., $X_2=X_3=\dots=X_n=0$. This event happens with probability $(1-2/n)^{n-1}$. Therefore, the probability that the policy wins with setting $\tau=1$ is $(1-2/n)^{2(n-1)}$.

    Let $k$ be an integer such that $2\le k\le n$. The threshold is set to $\tau=k$ when the index of the last success in the sample sequence is $k$, i.e., $Y_k=1$ and $Y_{k+1}=Y_{k+2}=\dots=Y_n=0$. This happens with probability $(2/n)\cdot(1-2/n)^{n-k}$. When $\tau=k$, the policy wins if and only if exactly one success is observed in $X_k,X_{k+1},\dots,X_n$. For each $\ell=k,k+1,\dots,n$, the probability that $X_{\ell}=1$ but $X_i=0$ for all $i\in\{k,k+1,\dots,n\}\setminus\{\ell\}$ is $(2/n)\cdot(1-2/n)^{n-k}$. Hence, the probability that the policy wins with setting $\tau=k$ is 
    \begin{align*}
    \frac{2}{n}\left(1-\frac{2}{n}\right)^{n-k}\cdot\sum_{\ell=k}^n \frac{2}{n}\left(1-\frac{2}{n}\right)^{n-k-1}
    =\frac{4(n-k+1)}{n^2}\left(1-\frac{2}{n}\right)^{2n-2k}.
    \end{align*}

    By summing up the probabilities of winning with each threshold index, we obtain the winning probability of the FLS policy. This probability is given by
    \begin{align*}
    \MoveEqLeft[2]
    \left(1-\frac{2}{n}\right)^{2n-2}+\sum\limits_{k=2}^n(n-k+1)\cdot\frac{2}{n}\cdot\frac{2}{n}\cdot\left(1-\frac{2}{n}\right)^{2n-2k}\\
    = {}& \left(1-\frac{2}{n}\right)^{\frac{n}{2}\cdot \frac{4(n-1)}{n}}+\frac{4}{n}\sum\limits_{k=2}^n \left(1-\frac{k-1}{n}\right)\left(1-\frac{2}{n}\right)^{\frac{n}{2}\cdot 4(1-\frac{k}{n})}\\
    \to{} & \frac{1}{e^4}+\int_0^1 4(1-x)\cdot\left(\frac{1}{e}\right)^{4(1-x)} \dd{x}=\frac{1}{e^4}+\int_0^1\frac{4x}{e^{4x}} \dd{x}\\
    = {}& \frac{1}{e^4}+\left[\frac{-(4x+1)}{4e^{4x}}\right]_0^1
    = \frac{1}{e^4}-\frac{5}{4e^4}+\frac{1}{4}=\frac{1-e^{-4}}{4},
    \end{align*}
    as $n$ goes to infinity, where the limit is evaluated by $\lim_{x\to\infty}(1-1/x)^x=1/e$ and interpreting the sum as a Riemann sum.
\end{proof}

\subsection{Direct utilization of success positions}\label{subsec:direct}
As demonstrated in the previous subsection, Bruss's policy with estimated success probabilities fails to achieve the winning probability of $1/4$. Nevertheless, by appropriately setting the threshold, we propose a \emph{deterministic} policy that guarantees a winning probability of $1/4$ as long as the sum of the odds $R$ is at least $(\sqrt{3}-1)/2$. Note that achieving a winning probability of $1/4$ is best possible according to Theorem~\ref{thm:NV}.

We consider stopping policies that determine the threshold index based only on the relative position from the success indices in the sample sequence. Then, the threshold needs to be set between the indices of the last success and the second-to-last success of the samples. Indeed, if the threshold is set after the last success index in the sample sequence, then the winning probability becomes zero for the instance with $n=2$ and $(p_1,p_2)=(1,0)$. Similarly, setting the threshold before or at the second last success index in the samples results in a winning probability of zero for the instance with $n=2$ and $(p_1,p_2)=(1,1)$. Setting the threshold at the last success (i.e., the FLS policy) is also not effective, as we showed in Theorem~\ref{thm:FLS}. Therefore, the threshold should be set between the indices of the last success and the second last success in the sample sequence.

We consider the policy of setting the threshold at the \emph{next} index of the second last success in the sample sequence. We refer to this policy as the \emph{After the Second Last Success (ASLS)} policy. If there are fewer than two successes in the samples, the threshold is assumed to be index $1$. We demonstrate that the ASLS policy wins with a probability of at least $1/4$ if the sum of the odds $R$ is at least $(\sqrt{3}-1)/2~(\approx0.3660)$.
\begin{thm}\label{thm:ASLS}
    The winning probability of the ASLS policy for the $(1,R)$-last success problem is at least $1/4$ if $R\ge (\sqrt{3}-1)/2$ and at least $\frac{R(4+3R)}{4(1+R)^2}$ if $0\le R\le (\sqrt{3}-1)/2$.
\end{thm}
\begin{proof}
We give the proof only for the cases where $0\le p_i<1$ for all $i\in[n]$. The other cases where $p_i=1$ for some $i\in[n]$ can be proved by considering the continuity of the winning probability with respect to $p_1,\dots,p_n$.

To simplify the analysis, we assume virtual trials where successes always occur at the $0$th and $-1$st positions, i.e., $p_0=p_{-1}=1$ and $X_0=X_{-1}=Y_0=Y_{-1}=1$.
With these virtual trials, both sample and actual sequences must contain at least two successes. Let $i_1$ and $i_2$ be the indices of the last and the second last success observed in the sample sequence $(Y_i)_{i=-1}^n$, respectively. Similarly, let $j_1$ and $j_2$ be the indices of the last and the second last success observed in $(X_i)_{i=-1}^n$, respectively. We then analyze the winning probability of the ASLS policy based on the relative positions of $i_1$, $i_2$, $j_1$, $j_2$.

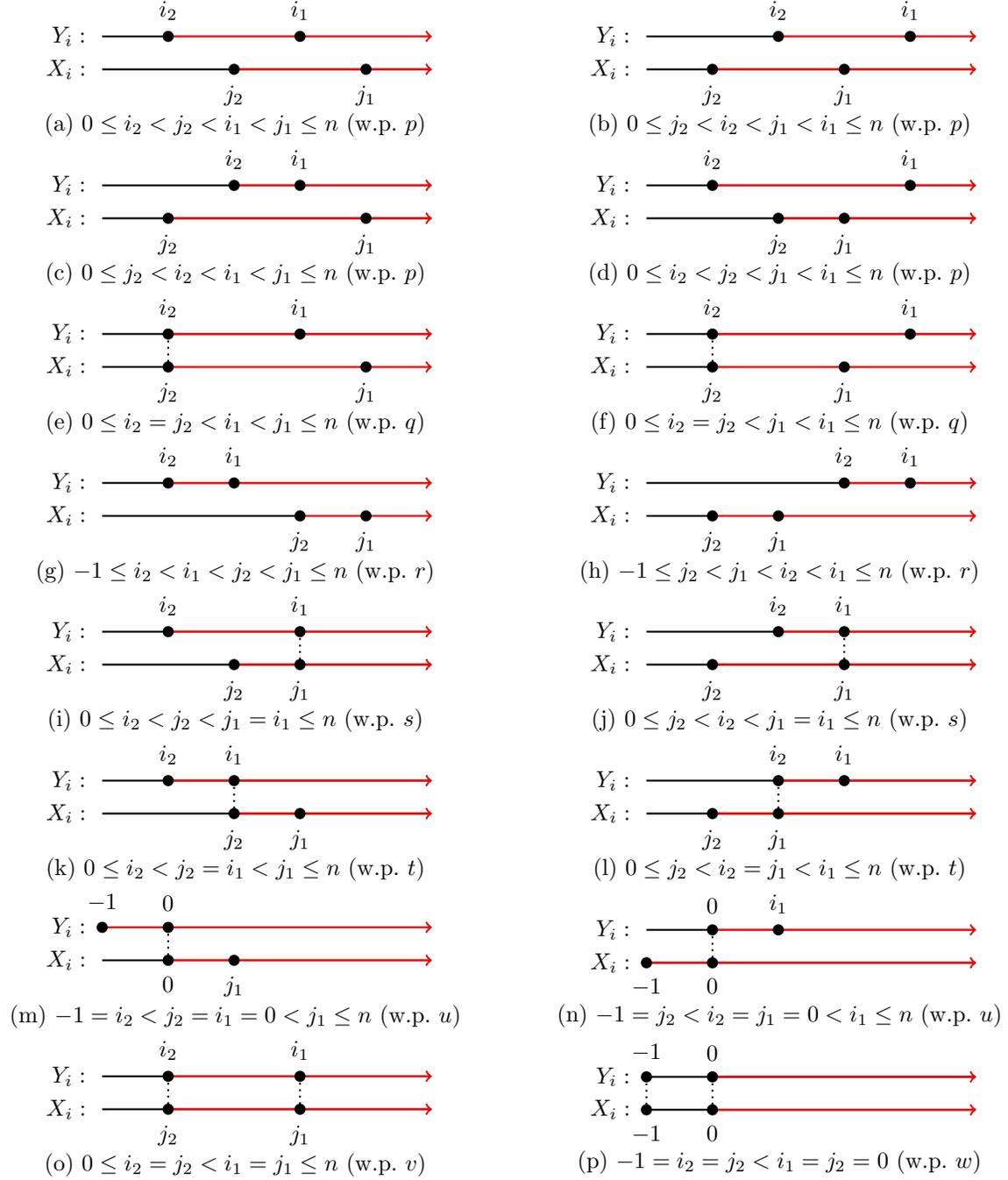
\begin{figure*}[!ht]
    \centering
    \begin{subfigure}[c]{0.5\textwidth}
    \centering
    \begin{tikzpicture}[classification]
        \draw[->](0,0.5) node [left=3pt]{$Y_i:$} -- (5,0.5);
        \draw[->](0,0) node [left=3pt]{$X_i:$} -- (5,0);
        \draw[->,red] (1,0.5) -- (5,0.5);
        \draw[->,red] (2,0) -- (5,0);
        \filldraw (1,0.5) circle (2pt) node [above=3pt]{$i_2$};
        \filldraw (3,0.5) circle (2pt) node [above=3pt]{$i_1$};
        \filldraw (2,0) circle (2pt) node [below=3pt]{$j_2$};
        \filldraw (4,0) circle (2pt) node [below=3pt]{$j_1$};
    \end{tikzpicture}
    \caption{$0\le i_2<j_2<i_1<j_1\le n$ (w.p.\ $p$)}
    \label{fig:i1<j1<i2<j2}
    \end{subfigure}%
    \begin{subfigure}[c]{0.5\textwidth}
    \centering
    \begin{tikzpicture}[classification]
        \draw[->](0,0.5) node [left=3pt]{$Y_i:$} -- (5,0.5);
        \draw[->](0,0) node [left=3pt]{$X_i:$} -- (5,0);
        \draw[->,red] (2,0.5) -- (5,0.5);
        \draw[->,red] (1,0) -- (5,0);
        \filldraw (2,0.5) circle (2pt) node [above=3pt]{$i_2$};
        \filldraw (4,0.5) circle (2pt) node [above=3pt]{$i_1$};
        \filldraw (1,0) circle (2pt) node [below=3pt]{$j_2$};
        \filldraw (3,0) circle (2pt) node [below=3pt]{$j_1$};
    \end{tikzpicture}
    \caption{$0\le j_2<i_2<j_1<i_1\le n$ (w.p.\ $p$)}
    \label{fig:j1<i1<j2<i2}
    \end{subfigure}%
    \par
    \begin{subfigure}[c]{0.5\textwidth}
    \centering
    \begin{tikzpicture}[classification]
        \draw[->](0,0.5) node [left=3pt]{$Y_i:$} -- (5,0.5);
        \draw[->](0,0) node [left=3pt]{$X_i:$} -- (5,0);
        \draw[->,red] (2,0.5) -- (5,0.5);
        \draw[->,red] (1,0) -- (5,0);
        \filldraw (2,0.5) circle (2pt) node [above=3pt]{$i_2$};
        \filldraw (3,0.5) circle (2pt) node [above=3pt]{$i_1$};
        \filldraw (1,0) circle (2pt) node [below=3pt]{$j_2$};
        \filldraw (4,0) circle (2pt) node [below=3pt]{$j_1$};
    \end{tikzpicture}
    \caption{$0\le j_2<i_2<i_1<j_1\le n$ (w.p.\ $p$)}
    \label{fig:j2<i2<i1<j1}
    \end{subfigure}%
    \begin{subfigure}[c]{0.5\textwidth}
    \centering
    \begin{tikzpicture}[classification]
        \draw[->](0,0.5) node [left=3pt]{$Y_i:$} -- (5,0.5);
        \draw[->](0,0) node [left=3pt]{$X_i:$} -- (5,0);
        \draw[->,red] (1,0.5) -- (5,0.5);
        \draw[->,red] (2,0) -- (5,0);
        \filldraw (1,0.5) circle (2pt) node [above=3pt]{$i_2$};
        \filldraw (4,0.5) circle (2pt) node [above=3pt]{$i_1$};
        \filldraw (2,0) circle (2pt) node [below=3pt]{$j_2$};
        \filldraw (3,0) circle (2pt) node [below=3pt]{$j_1$};
    \end{tikzpicture}
    \caption{$0\le i_2<j_2<j_1<i_1\le n$ (w.p.\ $p$)}
    \label{fig:i1<j1<j2<i2}
    \end{subfigure}
    \par
    \begin{subfigure}[c]{0.5\textwidth}
    \centering
    \begin{tikzpicture}[classification]
        \draw[->](0,0.5) node [left=3pt]{$Y_i:$} -- (5,0.5);
        \draw[->](0,0) node [left=3pt]{$X_i:$} -- (5,0);
        \draw[->,red] (1,0.5) -- (5,0.5);
        \draw[->,red] (1,0) -- (5,0);
        \draw[dotted] (1,0.5) -- (1,0);
        \filldraw (1,0.5) circle (2pt) node [above=3pt]{$i_2$};
        \filldraw (3,0.5) circle (2pt) node [above=3pt]{$i_1$};
        \filldraw (1,0) circle (2pt) node [below=3pt,fill=white]{$j_2$};
        \filldraw (4,0) circle (2pt) node [below=3pt]{$j_1$};
    \end{tikzpicture}
    \caption{$0\le i_2=j_2<i_1<j_1\le n$ (w.p.\ $q$)}
    \label{fig:i_2=j_2<i_1<j_1}
    \end{subfigure}%
    \begin{subfigure}[c]{0.5\textwidth}
    \centering
    \begin{tikzpicture}[classification]
        \draw[->](0,0.5) node [left=3pt]{$Y_i:$} -- (5,0.5);
        \draw[->](0,0) node [left=3pt]{$X_i:$} -- (5,0);
        \draw[->,red] (1,0.5) -- (5,0.5);
        \draw[->,red] (1,0) -- (5,0);
        \draw[dotted] (1,0.5) -- (1,0);
        \filldraw (1,0.5) circle (2pt) node [above=3pt]{$i_2$};
        \filldraw (4,0.5) circle (2pt) node [above=3pt]{$i_1$};
        \filldraw (1,0) circle (2pt) node [below=3pt,fill=white]{$j_2$};
        \filldraw (3,0) circle (2pt) node [below=3pt]{$j_1$};
    \end{tikzpicture}
    \caption{$0\le i_2=j_2<j_1<i_1\le n$ (w.p.\ $q$)}
    \label{fig:i_2=j_2<j_1<i_1}
    \end{subfigure}
    \par
    \begin{subfigure}[c]{0.5\textwidth}
    \centering
    \begin{tikzpicture}[classification]
        \draw[->](0,0.5) node [left=3pt]{$Y_i:$} -- (5,0.5);
        \draw[->](0,0) node [left=3pt]{$X_i:$} -- (5,0);
        \draw[->,red] (1,0.5) -- (5,0.5);
        \draw[->,red] (3,0) -- (5,0);
       \filldraw (1,0.5) circle (2pt) node [above=3pt]{$i_2$};
        \filldraw (2,0.5) circle (2pt) node [above=3pt]{$i_1$};
        \filldraw (3,0) circle (2pt) node [below=3pt]{$j_2$};
        \filldraw (4,0) circle (2pt) node [below=3pt]{$j_1$};
    \end{tikzpicture}
    \caption{$-1\le i_2<i_1<j_2<j_1\le n$ (w.p.\ $r$)}
    \label{fig:i_1<i_2<j_1<j_2}
    \end{subfigure}%
    \begin{subfigure}[c]{0.5\textwidth}
    \centering
    \begin{tikzpicture}[classification]
        \draw[->](0,0.5) node [left=3pt]{$Y_i:$} -- (5,0.5);
        \draw[->](0,0) node [left=3pt]{$X_i:$} -- (5,0);
        \draw[->,red] (3,0.5) -- (5,0.5);
        \draw[->,red] (1,0) -- (5,0);
        \filldraw (3,0.5) circle (2pt) node [above=3pt]{$i_2$};
        \filldraw (4,0.5) circle (2pt) node [above=3pt]{$i_1$};
        \filldraw (1,0) circle (2pt) node [below=3pt]{$j_2$};
        \filldraw (2,0) circle (2pt) node [below=3pt]{$j_1$};
    \end{tikzpicture}
    \caption{$-1\le j_2<j_1<i_2<i_1\le n$ (w.p.\ $r$)}
    \label{fig:j_1<j_2<i_1<i_2}
    \end{subfigure}
    \par
    \begin{subfigure}[c]{0.5\textwidth}
    \centering
    \begin{tikzpicture}[classification]
        \centering
        \draw[->](0,0.5) node [left=3pt]{$Y_i:$} -- (5,0.5);
        \draw[->](0,0) node [left=3pt]{$X_i:$} -- (5,0);
        \draw[->,red] (1,0.5) -- (5,0.5);
        \draw[->,red] (2,0) -- (5,0);
        \draw[dotted] (3,0.5) -- (3,0);
        \filldraw (1,0.5) circle (2pt) node [above=3pt]{$i_2$};
        \filldraw (3,0.5) circle (2pt) node [above=3pt]{$i_1$};
        \filldraw (2,0) circle (2pt) node [below=3pt]{$j_2$};
        \filldraw (3,0) circle (2pt) node [below=3pt,fill=white]{$j_1$};
    \end{tikzpicture}
    \caption{$0\le i_2<j_2<j_1=i_1\le n$ (w.p.\ $s$)}
    \label{fig:i_2<j_2<j_1=i_1}
    \end{subfigure}%
        \begin{subfigure}[c]{0.5\textwidth}
    \centering
    \begin{tikzpicture}[classification]
        \draw[->](0,0.5) node [left=3pt]{$Y_i:$} -- (5,0.5);
        \draw[->](0,0) node [left=3pt]{$X_i:$} -- (5,0);
        \draw[->,red] (2,0.5) -- (5,0.5);
        \draw[->,red] (1,0) -- (5,0);
        \draw[dotted] (3,0.5) -- (3,0);
        \filldraw (2,0.5) circle (2pt) node [above=3pt]{$i_2$};
        \filldraw (3,0.5) circle (2pt) node [above=3pt]{$i_1$};
        \filldraw (1,0) circle (2pt) node [below=3pt]{$j_2$};
        \filldraw (3,0) circle (2pt) node [below=3pt,fill=white]{$j_1$};
    \end{tikzpicture}
    \caption{$0\le j_2<i_2<j_1=i_1\le n$ (w.p.\ $s$)}
    \label{fig:j_2<i_2<j_1=i_1}
    \end{subfigure}
    \par
    \begin{subfigure}[c]{0.5\textwidth}
    \centering
    \begin{tikzpicture}[classification]
        \draw[->](0,0.5) node [left=3pt]{$Y_i:$} -- (5,0.5);
        \draw[->](0,0) node [left=3pt]{$X_i:$} -- (5,0);
        \draw[->,red] (1,0.5) -- (5,0.5);
        \draw[->,red] (2,0) -- (5,0);
        \draw[dotted] (2,0.5) -- (2,0);
        \filldraw (1,0.5) circle (2pt) node [above=3pt]{$i_2$};
        \filldraw (2,0.5) circle (2pt) node [above=3pt]{$i_1$};
        \filldraw (2,0) circle (2pt) node [below=3pt,fill=white]{$j_2$};
        \filldraw (3,0) circle (2pt) node [below=3pt]{$j_1$};
    \end{tikzpicture}%
    \caption{$0\le i_2<j_2=i_1<j_1\le n$ (w.p.\ $t$)}
    \label{fig:i_2<j_2=i_1<j_1}
    \end{subfigure}%
    \begin{subfigure}[c]{0.5\textwidth}
    \centering
    \begin{tikzpicture}[classification]
        \draw[->](0,0.5) node [left=3pt]{$Y_i:$} -- (5,0.5);
        \draw[->](0,0) node [left=3pt]{$X_i:$} -- (5,0);
        \draw[->,red] (2,0.5) -- (5,0.5);
        \draw[->,red] (1,0) -- (5,0);
        \draw[dotted] (2,0.5) -- (2,0);
        \filldraw (2,0.5) circle (2pt) node [above=3pt]{$i_2$};
        \filldraw (3,0.5) circle (2pt) node [above=3pt]{$i_1$};
        \filldraw (1,0) circle (2pt) node [below=3pt]{$j_2$};
        \filldraw (2,0) circle (2pt) node [below=3pt,fill=white]{$j_1$};
    \end{tikzpicture}
    \caption{$0\le j_2<i_2=j_1<i_1\le n$ (w.p.\ $t$)}
    \label{fig:j_2<i_2=j_1<i_1}
    \end{subfigure}
    \par
    \begin{subfigure}[c]{0.5\textwidth}
    \centering
    \begin{tikzpicture}[classification]
        \draw[->](0,0.5) node [left=3pt]{$Y_i:$} -- (5,0.5);
        \draw[->](0,0) node [left=3pt]{$X_i:$} -- (5,0);
        \draw[->,red] (0,0.5) -- (5,0.5);
        \draw[->,red] (1,0) -- (5,0);
        \draw[dotted] (1,0.5) -- (1,0);
        \filldraw (0,0.5) circle (2pt) node [above=3pt]{$-1$};
        \filldraw (1,0.5) circle (2pt) node [above=3pt]{$0$};
        \filldraw (1,0) circle (2pt) node [below=3pt,fill=white]{$0$};
        \filldraw (2,0) circle (2pt) node [below=3pt]{$j_1$};
    \end{tikzpicture}%
    \caption{$-1=i_2<j_2=i_1=0<j_1\le n$ (w.p.\ $u$)}
    \label{fig:-1=i_2<j_2=i_1<j_1}
    \end{subfigure}%
    \begin{subfigure}[c]{0.5\textwidth}
    \centering
    \begin{tikzpicture}[classification]
        \draw[->](0,0.5) node [left=3pt]{$Y_i:$} -- (5,0.5);
        \draw[->](0,0) node [left=3pt]{$X_i:$} -- (5,0);
        \draw[->,red] (1,0.5) -- (5,0.5);
        \draw[->,red] (0,0) -- (5,0);
        \draw[dotted] (1,0.5) -- (1,0);
        \filldraw (1,0.5) circle (2pt) node [above=3pt]{$0$};
        \filldraw (2,0.5) circle (2pt) node [above=3pt]{$i_1$};
        \filldraw (0,0) circle (2pt) node [below=3pt]{$-1$};
        \filldraw (1,0) circle (2pt) node [below=3pt,fill=white]{$0$};
    \end{tikzpicture}
    \caption{$-1=j_2<i_2=j_1=0<i_1\le n$ (w.p.\ $u$)}
    \label{fig:-1=j_2<i_2=j_1<i_1}
    \end{subfigure}
    \par
    \begin{subfigure}[c]{0.5\textwidth}
    \centering
    \begin{tikzpicture}[classification]
        \draw[->](0,0.5) node [left=3pt]{$Y_i:$} -- (5,0.5);
        \draw[->](0,0) node [left=3pt]{$X_i:$} -- (5,0);
        \draw[->,red] (1,0.5) -- (5,0.5);
        \draw[->,red] (1,0) -- (5,0);
        \draw[dotted] (1,0.5) -- (1,0);
        \draw[dotted] (3,0.5) -- (3,0);
        \filldraw (1,0.5) circle (2pt) node [above=3pt]{$i_2$};
        \filldraw (3,0.5) circle (2pt) node [above=3pt]{$i_1$};
        \filldraw (1,0) circle (2pt) node [below=3pt,fill=white]{$j_2$};
        \filldraw (3,0) circle (2pt) node [below=3pt,fill=white]{$j_1$};
    \end{tikzpicture}
    \caption{$0\le i_2=j_2<i_1=j_1\le n$ (w.p.\ $v$)}
    \label{fig:0<i_1=j_1}
    \end{subfigure}%
    \begin{subfigure}[c]{0.5\textwidth}
    \centering
    \begin{tikzpicture}[classification]
        \draw[->](0,0.5) node [left=3pt]{$Y_i:$} -- (5,0.5);
        \draw[->](0,0) node [left=3pt]{$X_i:$} -- (5,0);
        \draw[->,red] (1,0.5) -- (5,0.5);
        \draw[->,red] (1,0) -- (5,0);
        \draw[dotted] (1,0.5) -- (1,0);
        \draw[dotted] (0,0.5) -- (0,0);
        \filldraw (0,0.5) circle (2pt) node [above=3pt]{$-1$};
        \filldraw (1,0.5) circle (2pt) node [above=3pt]{$0$};
        \filldraw (0,0) circle (2pt) node [below=3pt,fill=white]{$-1$};
        \filldraw (1,0) circle (2pt) node [below=3pt,fill=white]{$0$};
    \end{tikzpicture}
    \caption{$-1=i_2=j_2<i_1=j_2=0$ (w.p.\ $w$)}
    \label{fig:v_2<v_1=0}
    \end{subfigure}%
    \caption{Classifications of the relative positions of $i_1,i_2,j_1,j_2$. Red lines indicate locations without success.}\label{fig:Classification}
\end{figure*}

We classify possible outcomes into 16 cases of (a)--(p) as illustrated in Figure~\ref{fig:Classification}. Let $p$ denote the probability that case (a) occurs. By symmetry, the probabilities that cases (b), (c), and (d) occur are also $p$ each. However, the probability that case (g) (and also (h)) occurs is different from $p$. This difference arises because the actual sequence in case (g) may contain successes between $i_1$ and $j_2$. Let us denote $q$ as the probability of case (e) occurring. By symmetry, the probability of case (f) occurring is $q$ as well. Additionally, let $r$ represent the probability of case (g) occurring, then the probability that case (h) occurs is also $r$ due to symmetry.

Consider a realization of the sample and actual sequences. If the realization is classified as being in case (g), then swapping the realizations of $X_{j_2}$ and $Y_{j_2}$ results in case (a), (c), or (e). Conversely, if the realization is classified as being in case (a), (c), or (e), then swapping the realizations of $X_{i_1}$ and $Y_{i_1}$ leads to case (g). As such swaps yield sequences with the same realization probability, we can conclude that
\begin{align}
r=2p+q. \label{eq:r=2p+q}
\end{align}
Similarly, let us denote the probability of case (i) (or case (j)) occurring as $s$ and the probability of case (k) (or case (l)) occurring as $t$, respectively. The relative position of case (m) matches that of case (k), but we categorize them separately since the ASLS policy wins for the case (m). We also consider the case (n) separately from the case (k) for the convenience of the analysis. Let the probability of cases (m) and (n) occurring be denoted as $u$. Moreover, let $v$ be the probability of cases (o) occurring. However, in the case of no success, we separate it and denote it as the case (p) with probability $w$ since the ASLS policy cannot win. Note that there are two or more non-virtual successes for cases (a) to (i) and case (o). In all other cases, non-virtual successes are at most one. Since each realization is classified uniquely into one case in Figure~\ref{fig:Classification}, we obtain
\begin{align}
4p+2q+2r+2s+2t+2u+v+w 
&=8p+4q+2s+2t+2u+v+w=1, \label{eq:total=1}
\end{align}
where the first equality holds by \eqref{eq:r=2p+q}.

Note that $w=\prod_{i=1}^n (1-p_i)^2$ represents the probability that successes occur only in virtual trials. The probability of case (m) occurring can be expressed as 
\begin{align}
u=\prod_{i=1}^n (1-p_i)^2\sum_{j=1}^n\frac{p_j}{1-p_j}=wR.\label{eq:u=wR}
\end{align}
Moreover, $w\cdot\left(\sum_{i=1}^n r_i\right)^2$ corresponds to the probability of having exactly one success in each of the actual and sample sequences, excluding the virtual trials. This situation is classified as case (e), (f), or (o). Thus, we have
\begin{align}
w\cdot\left(\sum_{i=1}^n r_i\right)^2=wR^2\le 2q+v. \label{eq:upper-u}
\end{align}

The probability of the last successes in both the actual and sample sequences having the same index ($i_1=j_1$) is 
\begin{align}
2s+v=\sum_{i=1}^n p_i^2\prod_{j=i+1}^n (1-p_j)^2, \label{eq:same-last}
\end{align}
where the right-hand side is obtained by directly computing the probability, and the left-hand side is derived from the probabilities that cases (i), (j), and (o) occur. The probability that the indices of the second last successes of the actual and the sample sequences are the same (i.e., $i_2=j_2$) is 
\begin{align}
    2q+v-w R^2
    =\sum_{i=1}^n p_i^2 \prod_{j=i+1}^n (1-p_j)^2 \Big(\sum_{j=i+1}^n r_{j}\Big)^2  \label{eq:same-second}
\end{align}
where the right-hand side is obtained by directly computing the probability, and the left-hand side is derived from the probabilities that cases (e), (f), and (o) occur. Moreover, the probability $t$ can be expressed as
\begin{align}
t=\sum_{i=1}^n p_i^2 \prod_{j=i+1}^n (1-p_j)^2 \left(\sum_{j=i+1}^n r_{j}\right). \label{eq:t}
\end{align}
By combining \eqref{eq:same-last}, \eqref{eq:same-second}, and \eqref{eq:t}, we have
\begin{align}
t
&\le\frac{(2s+v)+(2q+v-w\cdot R^2)}{2}
=q+s+v-\frac{w\cdot R^2}{2}, \label{eq:upper-t}
\end{align}
since $(1+(\sum_{j=i+1}^nr_{j})^2)/2\ge \sum_{j=i+1}^n r_{j}$ holds by the AM-GM inequality for every $i\in[n]$.
Since $w$ represents the probability of no success in both the sample and the actual sequence, we have
\begin{align}
    w\le\left(\frac{1}{1+R}\right)^2
    \label{eq:upper-w}
\end{align}
by Lemma~\ref{lem:nosuccess}.

Now, we are ready to demonstrate that the ASLS policy guarantees a winning probability of $1/4$. As depicted in Figure~\ref{fig:Classification}, the ASLS policy wins if the realization is classified as case (b), (c), (e), (f), (j), (m), or (o). Consequently, the winning probability is at least
\begin{align*}
    2p+2q+s+u+v
    &=\frac{8p+4q+2s+2t+2u+v+w}{4}+\frac{4q+2s-2t+2u+3v-w}{4} \hspace{-8mm}&\\
    &=\frac{1}{4}+\frac{4q+2s-2t+2u+3v-w}{4} &&(\text{by \eqref{eq:total=1}})\\
    &\ge\frac{1}{4}+\frac{4q+2s-2(q+s+v-w\cdot R^2/2)+2u+3v-w}{4} &&(\text{by \eqref{eq:upper-t}})\\
    &=\frac{1}{4}+\frac{2q+2u+v+w(R^2-1)}{4}\\
    &=\frac{1}{4}+\frac{2q+v+2wR+w(R^2-1)}{4} && (\text{by \eqref{eq:u=wR}})\\
    &\ge\frac{1}{4}+\frac{wR^2+2wR+w(R^2-1)}{4} &&(\text{by \eqref{eq:upper-u}})\\
    &=\frac{1}{4}+\frac{w(2R^2+2R-1)}{4}.
\end{align*}
Here, $(2R^2+2R-1)$ is non-negative if and only if $R\ge (\sqrt{3}-1)/2$. Thus, the winning probability is at least $1/4$ if $R$ is at least $(\sqrt{3}-1)/2$ by $w\ge 0$. Also, if $R<(\sqrt{3}-1)/2$, the winning probability is at least $\frac{1}{4}+\frac{2R^2+2R-1}{4(1+R)^2}=\frac{R(4+3R)}{4(1+R)^2}$ by \eqref{eq:upper-w}.
\end{proof}

For the $(1,R)$-last success problem, the winning probability of any policy is at most $1/4$ (Theorem~\ref{thm:NV}) and at most $R/e^R$ if $R\le 1$ (Proposition~\ref{prop:bruss}). Introducing $\alpha\approx0.357$ such that $\alpha/e^\alpha=1/4$, the upper bound is summarized as $R/e^R$ for $R\le \alpha$ and $1/4$ for $R\ge \alpha$. Consequently, the ASLS policy attains nearly the best possible winning probability for any $R$, as illustrated in Figure~\ref{fig:ASLS}. 

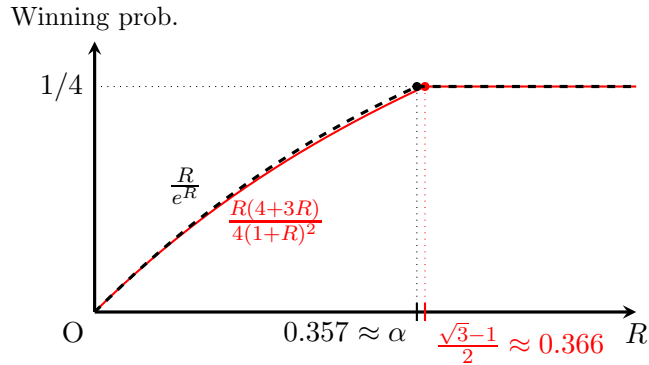
\begin{figure}[ht]
    \centering
    \begin{tikzpicture}[scale=12]
    \draw[->,>=stealth,very thick] (0,0) --(0.6,0) node[below]{$R$}; 
    \draw[->,>=stealth,very thick] (0,0)--(0,0.25) node[left]{$1/4$}--(0,0.3) node[above, font=\small]{Winning prob.}; 
    \draw (0,0) node[below left]{O};

    \draw[dotted] (0,0.25)--(0.6,0.25);

    \draw[domain=0:0.366, draw=red, thick] plot(\x, {1/4+(2*pow(\x,2)+2*\x-1)/(4+8*\x+4*pow(\x,2))});
    \draw[draw=red, thick] (0.366, 0.25)--(0.6,0.25);
    \draw[draw=red, dotted] (0.366,0.25)--(0.366,0) node[below right,red]{$\frac{\sqrt{3}-1}{2}\approx 0.366$};
    \draw[thick,red] (0.366,0.01) -- (0.366,-0.01);
    \fill[red] (0.366,0.25) circle (0.005);
    \node[red] at (0.2,0.1) {$\frac{R(4+3R)}{4(1+R)^2}$};

    \draw[domain=0.0001:0.357, very thick, dashed] plot(\x, \x/exp{\x});
    \draw[very thick, dashed] (0.357,0.25) -- (0.6,0.25);
    \draw[dotted] (0.357,0.25)--(0.357,0) node[below left] {$0.357\approx\alpha$};
    \draw[thick] (0.357,0.01) -- (0.357,-0.01);
    \fill (0.357,0.25) circle (0.005);
    \node at (0.1,0.14) {$\frac{R}{e^R}$};
    \end{tikzpicture}
    \caption{The winning probability of the ASLS policy (red line) and an upper bound of any policy for the $(1,R)$-last success problem (black dashed line).}
    \label{fig:ASLS}
\end{figure}

\subsection{Applying a policy by Nuti and Vondr\'{a}k}\label{subsec:NV}
In this subsection, we examine a randomized policy for the $(m,R)$-last success problem derived from the $1/4$-competitive policy for the adversarial order single sample secretary problem by Nuti and Vondr\'{a}k~\cite{NV23}. We demonstrate that the policy achieves a winning probability of $1/4$ only if $R$ is at least $1/2$. This means that the ASLS policy is superior in both performance and deterministic nature.

The adversarial order single sample secretary problem is described as follows: A sequence of independent real-valued random variables $V_1,V_2,\ldots,V_n$ is revealed one by one. Upon observing the value of $V_i$, an immediate and irrevocable decision must be made to either halt the observation or continue other subsequent. It is assumed that the distribution that each random variable $V_i$ follows is unknown, but a prior sample $W_i$ that follows the same distribution as $V_i$ is given for each $i\in[n]$. The goal is to maximize the probability of stopping at the highest value among $V_i$'s. The single sample last success problem can be reduced to the adversarial order single sample secretary problem by considering a distribution where $V_i=i$ with probability $p_i$ and $0$ with probability $1-p_i$, if we treat the no-success scenario as a win (or at least one success appears with probability $1$).

For the adversarial order single sample secretary problem, Nuti and Vondr\'{a}k~\cite{NV23} proposed the following simple policy: set the maximum value in the sample sequence as the threshold, and then stop when a value exceeding the threshold is observed. 
Here, if two values are the same, we make a random tie-break.\footnote{See, e.g., the paper of Rubinstein et al.~\cite{RWW20} for more details on this tie-break. Note that without such a random tie-break, the winning probability of this policy is at most $(1-e^{-4})/4$ by \ref{thm:FLS}.} This policy wins with a probability of at least $1/4$ as follows~\cite{NV23}: suppose the largest two numbers that appear in $V_1,\dots,V_n,W_1,\dots,W_n$ are $a_1,a_2$. Then, the policy wins if $a_2$ comes from $W_1,\dots,W_n$ and $a_1$ comes from $V_1,\dots,V_n$. Thus, it wins with probability $1/4$ if 
 $a_1,a_2$ come from different trials, and with probability $1/2$ if $a_1,a_2$ come from the same trial.

This policy leads to the following randomized policy for the single sample last success problem: set the threshold index $\tau$ either to the index of the last success or to the next index of the last success in the sample sequence, each with a probability of $1/2$. If no success appears in the sample sequence, it sets $\tau$ to be $1$. We refer to this policy as the \emph{From the Last Success, Randomized (FLSR)} policy. We demonstrate that the FLSR policy guarantees a winning probability of $1/4$ only if $R\ge 1/2$, which is a stronger condition than the one of the ASLS policy. We remark that the above elegant analysis conducted by Nuti and Vondr\'{a}k is not applicable to the FLSR policy if $R<\infty$. This is because a no-success scenario $(X_1=\dots=X_n=0)$ results in a loss in the last success problem, whereas its corresponding scenario ($V_1=\dots=V_n=0$) leads to a win in the reduced adversarial order secretary problem.
\begin{restatable}{thm}{FLSR}
    \label{thm:FLSR}
    The winning probability of the FLSR policy for the $(1,R)$-last success problem is $1/4$ when $R\ge 1/2$ and at least $\frac{1}{4}-\frac{1-2R}{4(1+R)^2}$ when $R\le 1/2$.
    Moreover, the winning probability is strictly less than $1/4$ if $R<1/2$.
\end{restatable}
\begin{proof}
Similar to the proof of Theorem~\ref{thm:FLSR}, we utilize the classification illustrated in Figure~\ref{fig:Classification}. The FLSR policy wins with probability $1$ if the realization is classified into cases (a), (c), (e), or (m). Moreover, it wins with probability $1/2$ if the realization is classified into cases (i), (j), (k), (m), or (o). As a result, the FLSR policy wins with probability at least
\begin{align*}
    2p+q+s+\frac{1}{2}t+u+\frac{1}{2}v&=\frac{8p+4q+2s+2t+2u+v+w}{4}+\frac{2s+2u+v-w}{4}\\
    &=\frac{1}{4}+\frac{2s+2u+v-w}{4} &&(\text{by \eqref{eq:total=1}})\\
    &=\frac{1}{4}+\frac{2s+v}{4}+\frac{2u-w}{4}\\
    &=\frac{1}{4}+\frac{2s+v}{4}+\frac{w(2R-1)}{4} &&(\text{by \eqref{eq:u=wR}}).
\end{align*}
Thus, the winning probability is at least $1/4$ if $R\ge 1/2$. If $R<1/2$, the winning probability is at least $\frac{1}{4}-\frac{1-2R}{4(1+R)^2}$ because $s,v\ge 0$ and $w\le 1/(1+R)^2$ by Lemma~\ref{lem:nosuccess}.
\end{proof}

\section{Multiple Sample Model}
In this section, we examine the last success problem with $m$ samples for general $m$. Let $\OPT(R)$ be the optimal winning probability of the last success problem with complete information about the distributions when the sum of the odds is at least $R$. As mentioned in Proposition~\ref{prop:bruss}, $\OPT(R)=R/e^R$ if $R\le 1$ and $\OPT(R)=1/e$ if $R\ge 1$. For any positive constant $\epsilon$, we propose a stopping policy that guarantees a winning probability of $\OPT(R)-\epsilon$ with a constant number of samples. Formally, we prove the following theorem.
\begin{thm}
    There exists a stopping policy that guarantees a winning probability of $\OPT(R)-O(1/\sqrt[4]{m})$ for the $(m,R)$-last success problem.
    \label{thm:mltsample}
\end{thm}
We emphasize that the number of samples required does not depend on the number of trials $n$. In addition, by Proposition~\ref{prop:incomplete}, no policy can guarantee a winning probability of $1/e$ with a finite number of samples.

We propose a threshold-based policy with such a winning probability. The threshold is determined by the samples as follows. For each $i\in[n]$, define $Q_i$ to be $\prod_{k=i}^n(1-p_k)$, which is the probability that the event $X_i=X_{i+1}=\dots=X_n=0$ occurs. For each $i\in [n]$ and $j\in [m]$, let $Y_{i,j}\in\{0,1\}$ be $j$th sample from $i$th distribution. The sequence $(Y_{1,j},Y_{2,j},\dots,Y_{n,j})$ is said to be $j$th sample sequence. For each $i\in[n]$, let $T_i$ be the number of sample sequences where no success is observed from $i$ to $n$, i.e., $T_i=|\{j\in[m]\mid Y_{i,j}=Y_{i+1,j}=\dots=Y_{n,j}=0\}|$. Since the event $Y_{i,j}=Y_{i+1,j}=\dots=Y_{n,j}=0$ happens with probability $Q_i$, the random variable $T_i$ follows a binomial distribution characterized by number $m$ and probability $Q_i$. We also define $\hat{Q}_i$ as $T_i/m$, which is an unbiased estimator of $Q$.

Let $\epsilon$ be a real such that $0<\epsilon<1/2$. Our policy first calculates the index $\hat{i}=\argmin\{i\in[n]\mid\hat{Q}_{i+1}\ge1/e+\epsilon\}$ from samples and then stops at the first success observed from $\hat{i}$ in the actual trials. Here, we assume that $\hat{Q}_{n+1}=1$. 

We prepare some lemmas to prove the theorem. Firstly, we show that the winning probability of a threshold-based policy is close to $1/e$ if the threshold is an index $i$ such that $Q_i$ is approximately $1/e$.
\begin{restatable}{lemma}{winprob}
    \label{lem:winprob}
    Let $i\in[n]$ and let $\delta$ be a positive real. If $Q_i\le1/e+\delta$ and $Q_{i+1}\ge1/e$, the winning probability of the stopping policy that stops at the first success from the index $i$ is at least $1/e-\delta$.
\end{restatable}
\begin{proof}
    By the assumption that $Q_{i+1}=\prod_{k=i+1}^n(1-p_k)\ge 1/e$, the probabilities $p_{i+1},p_{i+2},\dots,p_n$ are strictly less than $1$.
    As the policy wins if there is exactly one success in $X_i,X_{i+1},\dots,X_n$, 
    the winning probability is
    \begin{align*}
    \MoveEqLeft
        p_i\prod_{k=i+1}^n (1-p_k)+(1-p_i)\sum_{k=i+1}^n \left(\frac{p_k}{1-p_k}\prod_{\ell=i+1}^n(1-p_\ell)\right)
        =p_iQ_{i+1}+(1-p_i)Q_{i+1}R_{i+1}.
    \end{align*}
    By applying Lemma~\ref{lem:nosuccess} for $p_{i+1},p_{i+2},\dots,p_n$, we have $1/e^{R_{i+1}}\le Q_{i+1}$, and hence 
    \begin{align}
    R_{i+1}\ge -\log Q_{i+1}. \label{eq:RlogQ}
    \end{align}
    Moreover, by the assumption that $(1-p_i)Q_{i+1}=Q_i\le 1/e+\delta$, we have 
    \begin{align}
        p_iQ_{i+1}\ge Q_{i+1}-1/e-\delta. \label{eq:remove_pi}
    \end{align}
    Thus, the winning probability is at least 
    \begin{align*}
        p_iQ_{i+1}+(1-p_i)Q_{i+1}R_{i+1}
        &\ge p_iQ_{i+1}-(1-p_i)Q_{i+1}\log Q_{i+1}\\
        &=p_iQ_{i+1}\log (eQ_{i+1})-Q_{i+1}\log Q_{i+1}\\
        &\ge (Q_{i+1}-1/e-\delta)\log (eQ_{i+1})-Q_{i+1}\log Q_{i+1}\\
        &=Q_{i+1}-\frac{\log(eQ_{i+1})}{e}-\delta\log(eQ_{i+1})\\
        &\ge Q_{i+1}-\frac{\log(eQ_{i+1})}{e}-\delta\\
        &\ge\min_{x\in[1/e,\,1]}\left(x-\frac{\log(ex)}{e}\right)-\delta=\frac{1}{e}-\delta.
    \end{align*}
    Here, the first and second inequalities follow from \eqref{eq:RlogQ} and \eqref{eq:remove_pi}, respectively.
    The third and fourth inequalities hold by $Q_{i+1}\in [1/e,\, 1]$.
    The last equality comes from the fact that $x-\log(ex)/e$ is monotone increasing for $x\ge 1/e$, as its derivative is $1-1/(ex)$.
\end{proof}

Secondly, we evaluate the probability that $\hat{i}$ satisfies the condition of Lemma~\ref{lem:winprob} by showing that $\hat{Q}_i$ and $Q_i$ are close with high probability. Let $D_i=(Q_i-\hat{Q}_i)/Q_i$, which represents the relative difference of $\hat{Q}_i$ and $Q_i$. To bound the maximum deviation, we use Doob's inequality (see, e.g., \cite{D19}), which states that, for any $\lambda>0$ and any non-negative submartingale\footnotemark{} $S_1,S_2,\dots,S_n$, the following inequality holds:
\begin{align}
    \Pr[\max_{k=1}^nS_k\ge\lambda]\le \mathbb{E}[S_n]/\lambda.
    \label{eq:doob}
\end{align}
\footnotetext{A sequence of random variables $S_1,\dots,S_n$ is called
a \emph{martingale} if $\mathbb{E}[S_{i+1}\mid S_1,\dots,S_i]= S_i$ for all $i\in[n-1]$, and a \emph{submartingale} if $\mathbb{E}[S_{i+1}\mid S_1,\dots,S_i]\ge S_i$ for all $i\in[n-1]$.}

Let $i^*\in[n]$ be the minimum index $i$ such that $Q_{i+1}\ge 1/e$, where we assume that $Q_{n+1}=1$. Note that we have (i) $Q_1\le Q_{i^*}<1/e$ or (ii) $Q_{1}\ge 1/e$ and $i^*=1$. Additionally, we have $p_{i^*+1},p_{i^*+2},\ldots,p_n<1$ by $Q_{i^*+1}=\prod_{k=i^*+1}^n (1-p_k)\ge 1/e>0$. We show that $D_n\ldots,D_{i^*+1}$ is a martingale. 
\begin{restatable}{lemma}{submartingale}
    \label{lem:submartingale}
    The sequence $D_n,\ldots,D_{i^*+1}$ is a martingale. 
\end{restatable}
\begin{proof}
    Let $i\in\{i^*+1,\dots,n\}$. 
    By the definition of $\hat{Q}_i$, we obtain $\mathbb{E}[\hat{Q}_i\mid\hat{Q}_{i+1}]=(1-p_i)\hat{Q}_{i+1}$. Note that $Q_i$ for each $i$ is a constant determined by the instance. Thus, we have
    \begin{align*}
        \mathbb{E}[D_i\mid D_{i+1},\ldots,D_n]=\mathbb{E}[D_i\mid D_{i+1}]
        &=\frac{(1-p_i)Q_{i+1}-\mathbb{E}[\hat{Q}_i\mid\hat{Q}_{i+1}]}{(1-p_i)Q_{i+1}}\\
        &=\frac{(1-p_i)Q_{i+1}-(1-p_i)\hat{Q}_{i+1}}{(1-p_i)Q_{i+1}}
        =D_{i+1}.
    \end{align*}
\end{proof}
From this lemma, we can conclude that the sequence of random variables $\abs{D_n},\ldots,\abs{D_{i^*+1}}$ is a non-negative submartingale. By applying the Doob's inequality \eqref{eq:doob} to $\abs{D_n},\ldots,\abs{D_{i^*+1}}$, we can obtain the following inequality.
\begin{restatable}{lemma}{unionone}
\label{lem:unionone}
$\Pr[\max_{k=i^*+1}^nD_k\ge\epsilon]\le\sqrt{\frac{e}{m\epsilon^2}}$. 
\end{restatable}
\begin{proof}
    For each $i\in\{i^*+1,\dots,n\}$, the expected value $\mathbb{E}[\abs{D_i}]$ is at most
    \begin{align}
    \mathbb{E}\big[\abs{D_i}\big]
        &=\textstyle\mathbb{E}\left[\abs{\frac{Q_i-\hat{Q}_i}{Q_i}}\right]
        =\mathbb{E}\left[\abs{\frac{Q_i-T_i/m}{Q_i}}\right]
        =\frac{1}{mQ_i}\mathbb{E}\big[\abs{mQ_i-T_i}\big]
        =\frac{1}{mQ_i}\mathbb{E}\big[\sqrt{(mQ_i-T_i)^2}\big]\notag\\
        &\le\frac{1}{mQ_i}\sqrt{\mathbb{E}[(mQ_i-T_i)^2]}
        =\frac{1}{mQ_i}\sqrt{mQ_i(1-Q_i)}
        \le\sqrt{\frac{1}{mQ_i}}. \label{eq:abs}
    \end{align}
    Here, the first inequality follows from Jensen's inequality with the concave function $x\mapsto\sqrt{x}$. The last equality comes from the fact that the variance of $T_i$ is $mQ_i(1-Q_i)$, because $T_i$ follows a binomial distribution with parameters $m$ and $Q_i$. By incorporating this inequality \eqref{eq:abs} into Doob's inequality \eqref{eq:doob} with the non-negative submartingale $|D_n|,|D_{n-1}|\dots,|D_{i}|$, we obtain
    \begin{align*}
        \Pr[\max_{k=i}^nD_k\ge\epsilon]
        &\le \Pr[\max_{k=i}^n|D_k|\ge\epsilon]
        \le \frac{1}{\epsilon}\cdot\mathbb{E}\big[|D_i|\big]
        \le\sqrt{\frac{1}{m\epsilon^2Q_i}}.
    \end{align*}
    Specifically, by setting $i=i^*+1$, we have
    \begin{align*}
        \Pr[\max_{k=i^*+1}^nD_k\ge\epsilon]
        \le\sqrt{\frac{1}{m\epsilon^2 Q_{i^*+1}}}
        \le\sqrt{\frac{e}{m\epsilon^2}}
    \end{align*}
    The last inequality holds by $Q_{i^*+1}\ge 1/e$.
\end{proof}

Furthermore, we have the following lemma by Hoeffding's inequality~\cite{H94}.
\begin{restatable}{lemma}{uniontwo}
    $\Pr[\hat{Q}_{i^*}\ge\frac{1}{e}+\epsilon]\le\exp(-2\epsilon^2m)$ if $Q_1<1/e$.
    \label{lem:uniontwo}
\end{restatable}
\begin{proof}
    Note that by the assumption that $Q_1<1/e$, we have $Q_{i^*}<1/e$. We use Hoeffding's inequality~\cite{H94}, which says that for independent random variables $Z_1,Z_2,\ldots,Z_n$ such that $Z_i\in[a,b]~(\forall i\in[n])$ and for any positive real $t$, the following inequality holds:
    \begin{align*}
        \Pr[\frac{1}{n}\sum_{i=1}^n(Z_i-\mathbb{E}[Z_i])\ge t]\le\exp\left(-\frac{2nt^2}{(b-a)^2}\right).
    \end{align*}
    Specifically, the right-hand side of this inequality is $\exp(-2nt^2)$ when $a=0$ and $b=1$. Thus, we have
    \begin{align*}
        \Pr[\hat{Q}_{i^*}\ge\frac{1}{e}+\epsilon]&\le\Pr[\hat{Q}_{i^*}-\mathbb{E}[\hat{Q}_{i^*}]\ge\epsilon]
        \le\exp(-2\epsilon^2m),
    \end{align*}
    since $\hat{Q}_{i^*}$ can be interpreted as an average of $m$ Bernoulli random variables and $\mathbb{E}[\hat{Q}_{i^*}]=Q_{i^*}<1/e$.
\end{proof}

Next, we demonstrate that $\hat{i}$ satisfies the condition of Lemma~\ref{lem:winprob} if the events considered in Lemmas~\ref{lem:unionone} and~\ref{lem:uniontwo} occur.
\begin{restatable}{lemma}{boundPs}
    If $Q_1<1/e$, $\max_{k=i^*+1}^n D_k<\epsilon$ and $\hat{Q}_{i^*}< 1/e+\epsilon$, 
    then $Q_{\hat{i}+1}\ge 1/e$ and $Q_{\hat{i}}\le 1/e+3\epsilon$.
    \label{lem:boundPs}
\end{restatable}
\begin{proof}
    Suppose that $\max_{k=i^*+1}^n D_k<\epsilon$ and $\hat{Q}_{i^*}<1/e+\epsilon$. By definition, we have $Q_1\le Q_2\le\dots\le Q_n$ and $\hat{Q}_1\le \hat{Q}_2\le\dots\le \hat{Q}_n$. Since $\hat{Q}_{i^*}< 1/e+\epsilon$ and $\hat{Q}_{\hat{i}+1}\ge 1/e+\epsilon$, we have $\hat{i}\ge i^*$. Thus, we obtain $Q_{\hat{i}+1}\ge Q_{i^*+1}\ge 1/e$ by the definition of $i^*$.

    Next, we show that $Q_{\hat{i}}\le 1/e+3\epsilon$. Recall that $\hat{i}\ge i^*$. Since $Q_{i^*}< 1/e<1/e+3\epsilon$ by $Q_1<1/e$, we may assume that $\hat{i}\ge i^*+1$. As $(Q_{\hat{i}}-\hat{Q}_{\hat{i}})/Q_{\hat{i}}=D_{\hat{i}}\le \max_{k=i^*+1}^n D_k<\epsilon$, we have $Q_{\hat{i}}< \hat{Q}_{\hat{i}}/(1-\epsilon)$. Moreover, from $\hat{i}>i^*\ge 1$ and $\hat{i}\in\argmin\{i\in[n]\mid \hat{Q}_{i+1}\ge 1/e+\epsilon\}$, we have $\hat{Q}_{\hat{i}}<1/e+\epsilon$. Hence, we obtain
    \begin{align}
        Q_{\hat{i}}< \frac{\hat{Q}_{\hat{i}}}{1-\epsilon}<\frac{1/e+\epsilon}{1-\epsilon}<\left(\frac{1}{e}+\epsilon\right)(1+2\epsilon)=\frac{1}{e}+\left(\frac{2}{e}+1\right)\epsilon+2\epsilon^2<\frac{1}{e}+3\epsilon, \label{eq:Qhati-bound}
    \end{align}
    where the last two inequalities hold by $0<\epsilon<1/2$.
\end{proof}
Finally, we use the union bound to show that the probability of $\max_{k=1}^nD_k<\delta$ or $\hat{Q}_{i^*}<\frac{1}{e}+\delta$ occurring is small. Then, by combining this fact with Lemma~\ref{lem:winprob}, we prove Theorem~\ref{thm:mltsample}.
\begin{proof}[Proof of Theorem~\ref{thm:mltsample}]
We choose $\epsilon$ as the one that satisfies $m=e/\epsilon^4$, which leads to $\epsilon=O(1/\sqrt[4]{k})$. Since we perform an asymptotic analysis, we may assume that $0<\epsilon<1/2$. 

We first consider the case where $Q_1<1/e$. In this case, $Q_{i^*}<1/e$ by the definition of $i^*$. By Lemmas~\ref{lem:unionone} and \ref{lem:uniontwo}, $\max_{k=i^*+1}^n D_k\ge\epsilon$ or $\hat{Q}_{i^*}\ge1/e+\epsilon$ happens with probability at most $\sqrt{\frac{e}{m\epsilon^2}}+\exp(-2\epsilon^2m)$ by the union bound. Then, we have
\begin{align*}
    \sqrt{\frac{e}{m\epsilon^2}}+\exp(-2\epsilon^2m)
    &
    =\epsilon+\exp(-\frac{2e}{\epsilon^2})
    \le \epsilon+\frac{\epsilon^2}{2e}
    \le 2\epsilon,
\end{align*}
where the first inequality holds since $e^{-x}\le1/(1+x)\le1/x$ by $e^x\ge1+x$, and the last inequality holds by the assumption that $\epsilon<1/2$. Therefore, the probability that both $\max_{k=i^*+1}^n D_k<\epsilon$ and $\hat{Q}_{i^*}<1/e+\epsilon$ happen is at least $1-2\epsilon$. Under these conditions, we have $Q_{\hat{i}+1}\ge 1/e$ and $Q_{\hat{i}}\le 1/e+3\epsilon$ by Lemma~\ref{lem:boundPs}. Using such an index $\hat{i}$, the threshold policy wins with probability at least $1/e-3\epsilon$ by Lemma~\ref{lem:winprob}. Thus, the overall winning probability is 
\begin{align*}
(1-2\epsilon)\left(\frac{1}{e}-3\epsilon\right)\ge \frac{1}{e}-4\epsilon=\frac{1}{e}-O\left(\frac{1}{\sqrt[4]{m}}\right). 
\end{align*}

Next, suppose that $Q_1\ge 1/e$. In this case, we have $i^*=1$. Then, by Lemma~\ref{lem:unionone}, we have $\max_{k=2}^n D_k<\epsilon$ with probability at least $1-\sqrt{\frac{e}{m\epsilon^2}}=1-\epsilon$ by $m=e/\epsilon^4$. In what follows, we analyze the winning probability under the assumption that $\max_{k=2}^n D_k<\epsilon$. By the definition of $\hat{i}$, there are two cases where (i) $\hat{i}=1$ or (ii) $\hat{i}\ge 2$ and $\hat{Q}_{\hat{i}}<1/e+\epsilon$. If $\hat{i}=1$, the winning probability is $Q_1R$, which is at least $1/e$ if $R\ge 1$, and at least $R/e^R$ if $R<1$ by Lemma~\ref{lem:nosuccess}. If $\hat{i}\ge 2$ and $\hat{Q}_{\hat{i}}<1/e+\epsilon$, we have $(Q_{\hat{i}}-\hat{Q}_{\hat{i}})/Q_{\hat{i}}=D_{\hat{i}}<\epsilon$, and hence 
\begin{align*}
    Q_{\hat{i}}< \frac{\hat{Q}_{\hat{i}}}{1-\epsilon}<\frac{1/e+\epsilon}{1-\epsilon}<\left(\frac{1}{e}+\epsilon\right)(1+2\epsilon)=\frac{1}{e}+\left(\frac{2}{e}+1\right)\epsilon+2\epsilon^2<\frac{1}{e}+3\epsilon,
\end{align*}
where the last two inequalities hold by $0<\epsilon<1/2$. Thus, the winning probability is at least $1/e-3\epsilon$ by Lemma~\ref{lem:winprob}. In both cases (i) and (ii), the winning probability is at least $\OPT(R)-3\epsilon$. Further, as the assumption of $\max_{k=2}^n D_k<\epsilon$ happens with probability $1-\epsilon$, the overall winning probability is at least
\begin{align*}
(1-\epsilon)(\OPT(R)-3\epsilon)
&\ge \OPT(R)-4\epsilon
=\OPT(R)-O\left(\frac{1}{\sqrt[4]{m}}\right).
\end{align*}
\end{proof}

\section*{Acknowledgments}
This work was supported by JST SPRING Grant Number JPMJSP2108, JST ERATO Grant Number JPMJER2301, JST PRESTO Grant Number JPMJPR2122, JSPS KAKENHI Grant Number JP20K19739, and Value Exchange Engineering, a joint research project between Mercari, Inc.\ and the RIISE.
\bibliographystyle{abbrv}
\bibliography{main}
\end{document}